\documentclass[reqno]{amsart}
\usepackage{amsmath}
\usepackage{amssymb, amsthm, amsfonts, tikz-cd, mathrsfs,mathtools,stmaryrd,enumitem, algorithmic,float,comment,tikz,hyperref}
\usepackage[toc]{appendix}
\usetikzlibrary{backgrounds}
\usetikzlibrary{decorations.pathreplacing}
\usepackage{fullpage}
\usepackage{xcolor}

\usepackage{tikz}
\usetikzlibrary{calc,decorations.pathreplacing}

\hypersetup{
  colorlinks   = true, 
  urlcolor     = blue, 
  linkcolor    = blue, 
  citecolor   = red 
}

\usepackage{comment}

\usepackage[capitalize,nameinlink]{cleveref}
\crefname{theorem}{Theorem}{Theorems}
\crefname{lemma}{Lemma}{Lemmas}
\crefname{claim}{Claim}{Claims}
\crefname{prop}{Proposition}{Propositions}
\crefname{figure}{Figure}{Figures}

\usepackage{bbm}
\newtheorem{theorem}{Theorem}

\newtheorem{question}[theorem]{Question}

\newtheorem{lemma}[theorem]{Lemma}
\newtheorem{claim}[theorem]{Claim}

\newtheorem{conj}[theorem]{Conjecture}

\newtheorem*{claim*}{Claim}

\newtheorem*{remark*}{Remark}
\newtheorem{remark}[theorem]{Remark}

\theoremstyle{definition}
\newtheorem{definition}[theorem]{Definition}

\numberwithin{theorem}{section}

\renewcommand{\phi}{\varphi}

\renewcommand{\leq}{\le}
\renewcommand{\geq}{\ge}

\newcommand{\cV}{\mathcal V}
\newcommand{\E}{\mathbb{E}}

\def\1{\mathbbm{1}}

\newcommand{\FF}{\mathbb F}
\newcommand{\RR}{\mathbb R}

\def\eqdef{\overset{\rm{def}}{=}}

\renewcommand{\le}{\leqslant}
\renewcommand{\ge}{\geqslant}

\newcommand{\EE}{\mathbb E}

\newcommand{\PP}{\mathbb P}

\newcommand{\floor}[1]{\lfloor#1\rfloor}

\newcommand{\abs}[1]{\left|#1\right|}

\newcommand{\F}{\mathbb{F}}

\newcommand{\Ber}{\textup{Ber}}
\newcommand{\rank}{\textup{rank}}

\newcommand{\OffDiag}{\textup{OffDiag}}

\newcommand{\Span}{\textup{span}}
\newcommand{\Coef}{\textup{Coef}}
\newcommand{\dtv}{d_{\textup{TV}}}
\usepackage[
backend=biber,
style=numeric,
maxnames=99,
giveninits=true
]{biblatex}
\renewbibmacro{in:}{%
  \ifentrytype{article}{}{\printtext{\bibstring{in}\intitlepunct}}}

\title
{Vertex-minor universality of a random graph} 

\makeatletter
\newcommand\thankssymb[1]{\textsuperscript{\@fnsymbol{#1}}}
\makeatother

\author[Ting-Wei Chao]{Ting-Wei Chao\thankssymb{1}}
\author[Zixuan Xu]{Zixuan Xu\thankssymb{1}}
\thanks{Department of Mathematics, Massachusetts Institute of Technology. Email: \texttt{\{twchao,zixuanxu\}@mit.edu}}

\begin{document}

\begin{abstract}
    Given a graph $G$ and a vertex $v\in V(G)$, a local complementation at $v$ on $G$ is an operation that replaces the induced graph on the neighborhood of $v$ by its complement. A graph $H$ is a vertex-minor if $H$ can be obtained from $G$ by a sequence of vertex deletions and local complementation. A graph is said to be $k$-vertex-minor universal if it contains every $k$-vertex graph on any $k$-subset of vertices as a vertex minor. Previously, Ascoli--Frederickson--Frederickson--McFarland--Post~\cite{AFFMP26} proved that with high probability $G(n,1/2)$ is $\Omega(\sqrt{n})$-vertex-minor universal. Furthermore, they conjectured that with high probability $G(n,p)$ and $G(n,1-p)$ are $\Omega(p\sqrt{n})$-vertex-minor universal for all $\omega(1/\sqrt{n})= p\le 1/2$. In this short note, we confirm this conjecture up to an extra logarithm factor and show that this is true with probability $1-2^{-\Omega(p^2n)}$ if $\Omega(\log n/\sqrt{n})= p\le 1/2$. 
    Together with a complementary result which applies to the regime where $1/\sqrt{n}\le p\le n^{-1/3}$ produced by an internal model at OpenAI, the conjecture is fully confirmed.
    
\end{abstract}

\maketitle

\section{Introduction}

Let $G = (V,E)$ be a graph. For a vertex $v\in V$, a \emph{local complementation} at $v$ is an operation where we replace the neighborhood of $v$ by its complement. Let $G * v$ denote the resulting graph after performing local complementation at $v$. We say that two graphs are \emph{locally equivalent} if one can be obtained from the other by a sequence of local complementations. For a graph $H$ with vertex set $V(H)\subseteq V$, we say that $H$ is a vertex-minor of $G$ if $H$ can be obtained from $G$ via a sequence of vertex deletions and local complementations. Equivalently, $H$ is a vertex-minor of $G$ if $H$ is the induced subgraph on $V(H)$ of a graph locally equivalent to $G$. Since their introduction in works of Bouchet~\cite{Bouchet1988Isotropic} and Fon-Der-Flaass~\cite{FonDerFlaass1988Local} in the 1980s, vertex-minors have been extensively studied in structural graph theory. In particular, local complementations preserve the cut-rank function of a graph. Thus, vertex-minors have close connections to the study of circle graphs, forbidden minors and related decomposition questions. See the survey on vertex-minors by Kim and Oum~\cite{KimOum2024VertexMinorsSurvey} for a more detailed overview of vertex-minors and related concepts.

Motivated by the connections between vertex-minors and quantum communication, a recent line of work considers the universality question for vertex-minors. A graph $G$ is \emph{$k$-vertex-minor universal} if for every vertex subset $U\subseteq V(G)$ of size $k$, every graph on $U$ is a vertex-minor of $G$. Note that this is a strong condition where every graph on $k$ vertices must be realized on any specified subset of $k$ vertices in $G$. In this direction, Claudet, Mhalla, and Perdrix~\cite{ClaudetMhallaPerdrix2023} first showed that there exists $k$-vertex-minor universal graphs with $O(k^4\log k)$ vertices and any $k$-vertex-minor universal graph must contain at least $(k-2)^2/(2\log_2 3)$ vertices. Then, Cautr\'es, Claudet, Mhalla, Perdrix, Savin, and Thomass{\'e}~\cite{CautresClaudetMhallaPerdrixSavinThomasse2024} further showed the existence of $k$-vertex-minor universal graphs with $(2+o(1))k^2$ vertices. Most recently, Ascoli, Frederickson, Frederickson, McFarland, and Post~\cite{AFFMP26} showed that almost all graphs on $n$ vertices are $O(\sqrt{n})$-vertex-minor universal by proving the following statement. Recall that $G(n,p)$ denotes the Erd\H{o}s-Reny\'i random graph on $n$ vertices with each edge present with probability $p$.

\begin{theorem}[\cite{AFFMP26}]\label{thm:1/2}
    For a positive integer $k$ and $G\sim G(n,1/2)$, if $n\ge (1+C)\frac{1}{2\log_2(4/3)}k^2$ for some $C>0$, then $G$ is $k$-vertex-minor universal with probability at least $1-2^{-(1+o(1))Ck^2/2}$.
\end{theorem}

Furthermore, the authors conjecture that an analogous result holds for $G(n,p)$ and $G(n,1-p)$ where $p=\omega(1/\sqrt{n})$ and $p\le 1/2$. Formally, they conjecture the following statement. 

\begin{conj}[\cite{AFFMP26}]\label{conj:all-p}
    If $p = \omega(1/\sqrt{n})$ and $p\le 1/2$ and $G\sim G(n,p)$ or $G\sim G(n,1-p)$, then $G$ is $k$-vertex-minor universal with high probability for some $k = \Omega(p\sqrt{n})$.
\end{conj}

It was remarked in \cite{AFFMP26} that the method used to prove \cref{thm:1/2} cannot be naturally extended to other values of $p\ne 1/2$. This is because \cref{thm:1/2} is proved by considering the random walk on the set of graphs on a fixed subset $U$ of size $k$ produced by a random sequence of local complementations at vertices outside of $U$. A crucial property used in the proof is the fact that each step of the random walk is independent when $p = 1/2$, which no longer holds for $p\ne 1/2$. Our main result is to confirm \cref{conj:all-p} up to an extra logarithm factor.

\begin{theorem}\label{thm:main}
    There exists a constant $n_0\in\mathbb{N}$ such that the following holds for all $n>n_0$.
    Suppose that $p\in[0,1]$ and $q\eqdef \min(p,1-p)\geq 100\log n/\sqrt{n}$. For any $k\leq q\sqrt{n}/100$, and $G\sim G(n,p)$, the probability that $G$ is $k$-vertex-minor universal is at least $1-2^{-q^2n/100}$.
\end{theorem}

Our main technical contribution is \cref{lem:close-to-uniform}, which formalizes the following intuition: Let $G\sim G(n,p)$ and fix a set of $\Omega(k^2)$ vertices $U\subseteq V(G)$. Let $W \eqdef V(G)\setminus U$. Then after a sequence of local complementations at $\{w_1,\dots, w_r\}= W$, the distribution of the induced subgraph $(G * w_1 * \dots * w_r)[U]$ should be close to $G(U,1/2)$ in total variation distance. Indeed, if $W$ is an independent set, any edge $uv$ with $u,v\in U$ is toggled during local complementation at $w_i$ if $u,v\in N(w_i)$. Furthermore, since $W$ is an independent set, local complementation at $w_j$ for $j < i$ does not toggle the edges between $w_i$ and $U$. So, the event that $u,v\in N(w_i)$ holds with probability $p^2$ and only depends on the edges between $w_i$ and $U$. Moreover, the events that $u,v\in N(w_i)$ are independent for all $w_i$ where $i\in [r]$. Therefore, by a direct calculation, the probability that $uv$ is an edge in $(G * w_1 * \dots * w_r)[U]$ is $\frac{1}{2}\pm 2^{-\Omega(q^2r)}$. The main difficulty is to generalize this observation from the distribution over a single edge $uv$ to the distribution over all possible graphs on $U$, and to deal with the dependencies between local complementations at $w_1,\dots,w_r$ when $W$ is not an independent set in $G$.
Once we have proved \cref{lem:close-to-uniform}, we can prove \cref{thm:main} by applying \cref{thm:1/2} to $(G * w_1 * \dots * w_r)[U]$ in a black-box way. 

\medskip \noindent\textbf{Update.} After we put our previous version of this note on arXiv, the following complementary result solving \cref{conj:all-p} in the regime where $p$ is close to $0$ or $1$ was generated by an internal model of OpenAI, and communicated to us by Mehtaab Sawhney.

\begin{theorem}\label{thm:openai}
For $n$ sufficiently large, let $p\in (0,1)$ satisfy $1/\sqrt{n}\leq p\leq n^{-1/3}/2$ and set $k = \floor{p\sqrt n}$. Then the following statements hold with probability at least $1-2e^{-pn/200}$:
\begin{enumerate}[label=(\alph*)]
    \item if $G \sim G(n,p)$, then $G$ is $k$-vertex-minor universal;\label{thm:openai_part_a}
    \item if $G \sim G(n,1-p)$, then $G$ is $k$-vertex-minor universal.\label{thm:openai_part_b}
\end{enumerate}
\end{theorem}
The above theorem is proved using an argument fundamentally different from ours, which only applies when the graph is sufficiently sparse or sufficiently dense. We will present the proof of \cref{thm:openai} in \cref{appendix:openai}.

Combining \cref{thm:main} and \cref{thm:openai}, the conjecture by Ascoli--Frederickson--Frederickson--McFarland--Post~\cite{AFFMP26} is now confirmed in all regime.

\medskip \noindent\textbf{Notations.} For a graph $G$ and a subset $U\subseteq V(G)$, we use $G[U]$ to denote the induced subgraph of $G$ on $U$. For vertex sets $A,B\subseteq V(G)$, we use $G[A,B]$ to denote the induced bipartite graph of $G$ on $A\times B$. For sets $A$ and $B$, we denote their symmetric difference as $A\triangle B \eqdef (A\setminus B)\cup (B\setminus A)$. For $p\in [0,1]$, we use $x\sim \Ber(p)$ to denote that $x$ is a Bernoulli random variable taking value $1$ with probability $p$ and $0$ with probability $1-p$. For any matrix $A$, we use $A_{ij}$ to denote the $(i,j)$-entry of $A$. For an $m\times n$ matrix $A$ and a $m'\times n'$ matrix $B$, the tensor product $A\otimes B$ is the $(mm')\times(nn')$ matrix indexed by pairs $(i,i')\in [m]\times[m']$ and $(j,j')\in [n]\times[n']$ with entries $(A\otimes B)_{(i,i'),(j,j')} = A_{ij}B_{i'j'}$. For a matrix $M\in \F_2^{m\times n}$, we use $\rank(M)$ to denote its rank over $\F_2$.

\medskip \noindent \textbf{Paper Organization.} In \cref{sec:prelim}, we present a preliminary lemma that will be used in proving \cref{thm:main}. Then in \cref{sec:proof}, we give the proof of \cref{thm:main}. In \cref{sec:conclude}, we include some concluding remarks about applying our techniques to other related settings. In \cref{appendix:openai}, we include the proof of \cref{thm:openai}. 

\section{Preliminary}\label{sec:prelim}

In this section, we present useful definitions and lemmas that will be used in our proof of \cref{thm:main}. We start with presenting a useful lemma that gives an upper bound for $\abs{\E[(-1)^{f(x_1,\dots, x_m)}}$ where $x_1,\dots, x_m$ are independent $\Ber(p)$ random variables and $f(x_1,\dots,x_m)$ is a multilinear polynomial of degree at most $2$ over $\F_2$ in terms of the rank of the coefficient matrix of degree $2$ terms in $f$. For any multilinear polynomial $f\in\FF_2[x_1,\dots,x_m]$ with $\deg f\leq 2$, we use $\Coef_2(f)$ to denote the $m\times m$ coefficient matrix over $\F_2$ with entries $(a_{ij})_{i,j=1}^m$ that records the degree $2$ coefficient of $f$. Namely, $a_{ij}=1$ if $x_ix_j$ is a monomial in $f$ and $a_{ij}=0$ otherwise. In particular, since $f$ is multilinear, every entry on the diagonal of $\Coef_2(f)$ is $0$. 

\begin{lemma}\label{lem:1}
    Let $p\in [0,1]$ and $q\eqdef\min(p,1-p)$. Let $r$ be a positive integer and $X\in \F_2^m$ be a vector with independent $\Ber(p)$ entries. Let $f\in\FF_2[x_1,\dots,x_m]$ be a multilinear polynomial with $\deg f\leq 2$. Suppose that $\rank(\Coef_2(f))=r$. Then we have
    \[\abs{\E[(-1)^{f(X)}]}\le (1-2q^2)^{\lfloor r/6\rfloor}.\]
\end{lemma}

\begin{proof}
    We prove by induction on $r$. In the base case where $r\leq 6$, the statement is trivial since the right hand side is $1$.

    Now suppose the statement is true for all integers less than $r$ and $r\geq 6$. Let $X=(X_1,\dots,X_m)\in \F_2^m$ be a vector with independent $\Ber(p)$ entries.
    Let $f\in\FF_2[x_1,\dots,x_m]$ be a multilinear polynomial with $\deg f\leq 2$ such that $\rank(\Coef_2(f))=r$. Let $\Coef_2(f)=(a_{ij})_{i,j=1}^m$ and we can further assume without loss of generality that $a_{12}=1$. Then we can write $f$ as
    \[f=x_1x_2+x_1f_1+x_2f_2+f_3,\]
    where $f_1,f_2,f_3\in \FF_2[x_3,\dots,x_m]$ such that $\deg f_1,\deg f_2\leq 1$, $\deg f_3\leq 2$, and $f_3$ is multilinear. Note that $\Coef_2(f_3)$ is a minor of $\Coef_2(f)$ obtained from removing the first two rows and columns from $\Coef_2(f)$. Therefore, we have $\rank(\Coef_2(f_3))\geq r-4$. Denote $X'=(X_3,\dots,X_m)$ and we have 
    \[\EE[(-1)^{f(X)}] = \EE_{X'}\left[(-1)^{f_3(X')}\EE_{X_1,X_2}[(-1)^{X_1X_2+X_1f_1(X')+X_2f_2(X')}\mid X']\right].\]

    We first prove the case $q=p\leq 1/2$.
    By a direct calculation, for any $a,b\in\FF_2$ and since $X_1,X_2\sim \Ber(p)$ are independent, we have
    \begin{align}\label{eq:Exp}
        \EE_{X_1,X_2}[(-1)^{X_1X_2+aX_1+bX_2}] = (1-p)^2 + p(1-p)((-1)^a+(-1)^b)-p^2(-1)^{a+b}.
    \end{align}
    Using the identity that 
    \[(-1)^{a+b}= 1+((-1)^a + (-1)^b)-2(-1)^{ab},\]
    we obtain
    \begin{equation}
        \EE_{X_1,X_2}[(-1)^{X_1X_2+aX_1+bX_2}] = (1-2p) + p(1-2p)((-1)^a+(-1)^b)+2p^2(-1)^{ab}.
    \end{equation}
    Therefore, we can write
    \begin{align*}
        \EE[(-1)^{f(X)}] =&\,\EE_{X'}\left[(-1)^{f_3(X')}\left((1-2p) + p(1-2p)((-1)^{f_1(X')}+(-1)^{f_2(X')})+2p^2(-1)^{f_1(X')f_2(X')}\right)\right]\\
        =&\,(1-2p)\EE_{X'}\left[(-1)^{f_3(X')}\right]+p(1-2p)\EE_{X'}\left[(-1)^{f_3(X')+f_1(X')}\right]\\
        &+p(1-2p)\EE_{X'}\left[(-1)^{f_3(X')+f_2(X')}\right]+2p^2\EE_{X'}\left[(-1)^{f_3(X')+f_1(X')f_2(X')}\right].
    \end{align*}
    Now, we can apply the inductive hypothesis to bound each term. Recall that we have $\rank(\Coef_2(f_3))\geq r-4$. Therefore, we obtain 
    \[\abs{\EE_{X'}\left[(-1)^{f_3(X')}\right]}\leq (1-2p^2)^{\lfloor(r-4)/6\rfloor}\leq (1-2p^2)^{\lfloor r/6\rfloor-1}.\]
    Furthermore, since $\deg f_1, \deg f_2\le 1$, we have $\Coef_2(f_3+f_1)=\Coef_2(f_3+f_2)=\Coef_2(f_3)$, which has rank at least $r-4$. Thus, we can conclude similarly that
    \[\abs{\EE_{X'}\left[(-1)^{f_3(X')+f_1(X')}\right]}\leq (1-2p^2)^{\lfloor r/6\rfloor-1}\text{ and }\abs{\EE_{X'}\left[(-1)^{f_3(X')+f_2(X')}\right]}\leq (1-2p^2)^{\lfloor r/6\rfloor-1}.\]
    Finally, note that the homogeneous degree $2$ part of $f_1f_2$ is $\sum_{i,j=3}^m\alpha_i\beta_jx_ix_j$ if $f_1=\sum_{i=3}^m\alpha_ix_i+c_1$, and $f_2=\sum_{j=3}^m\beta_jx_j+c_2$. Consider the multilinearization $g$ of $f_1f_2$, i.e. the polynomial $g$ obtained by replacing all the terms $x_i^2$ by $x_i$ in $f_1f_2$. It follows that the coefficient of $x_ix_j$ in $g$ is $\alpha_i\beta_j+\alpha_j\beta_i$ if $i\neq j$. Therefore, we have
    \[\Coef_2(g)=\alpha^T\beta+\beta^T\alpha\]
    where $\alpha=(\alpha_3,\dots,\alpha_m)$ and $\beta=(\beta_3,\dots,\beta_m)$. Since the diagonal terms of $\alpha^T\beta+\beta^T\alpha$ are always $0$, they agree with the diagonal terms of $\Coef_2(g)$.
    Therefore, we have $\rank(\Coef_2(g))\leq 2$. Hence, we obtain that $\rank(\Coef_2(f_3+g))\geq r-6$. By the inductive hypothesis, we have
    \[\abs{\EE_{X'}\left[(-1)^{f_3(X')+f_1(X')f_2(X')}\right]}=\abs{\EE_{X'}\left[(-1)^{f_3(X')+g(X')}\right]}\leq (1-2p^2)^{\lfloor(r-6)/6\rfloor} = (1-2p^2)^{\lfloor r/6\rfloor-1}.\]
    Combining the above inequalities and the fact that $p,(1-2p)\geq 0$ for $p\le 1/2$, we obtain
    \[\abs{\EE[(-1)^{f(X)}]}\leq \left((1-2p)+p(1-2p)+p(1-2p)+2p^2\right)(1-2p^2)^{\lfloor r/6\rfloor-1}=(1-2p^2)^{\lfloor r/6\rfloor}.\]

    For the case $p\geq 1/2$ and $q=1-p$, we can combine \cref{eq:Exp} with the identity
    \[1=(-1)^{a+b}-((-1)^a + (-1)^b)+2(-1)^{ab}\]
    and conclude that
    \begin{align*}
        \EE_{X_1,X_2}[(-1)^{X_1X_2+aX_1+bX_2}] =&-(2p-1)(-1)^{a+b}+ (1-p)(2p-1)((-1)^a+(-1)^b)+2(1-p)^2(-1)^{ab}\\
        =& -(1-2q)(-1)^{a+b}+q(1-2q)((-1)^a+(-1)^b)+2q^2(-1)^{ab}.      
    \end{align*}
    Similar as before, we have
    \begin{align*}
        &\EE[(-1)^{f(X)}]\\
        &=\EE_{X'}\left[(-1)^{f_3(X')}\left(-(1-2q)(-1)^{f_1(X')+f_2(X')} + q(1-2q)((-1)^{f_1(X')}+(-1)^{f_2(X')})+2q^2(-1)^{f_1(X')f_2(X')}\right)\right]\\
        &=-(1-2q)\EE_{X'}\left[(-1)^{f_3(X')+f_1(X')+f_2(X')}\right]+q(1-2q)\EE_{X'}\left[(-1)^{f_3(X')+f_1(X')}\right]\\
        &\quad+q(1-2q)\EE_{X'}\left[(-1)^{f_3(X')+f_2(X')}\right]+2q^2\EE_{X'}\left[(-1)^{f_3(X')+f_1(X')f_2(X')}\right].\\
        &\leq \left((1-2q)+q(1-2q)+q(1-2q)+2q^2\right)(1-2q^2)^{\lfloor r/6\rfloor-1}=(1-2q^2)^{\lfloor r/6\rfloor},
    \end{align*}
    where the last inequality follows from applying the inductive hypothesis to $f_3+f_1+f_2$, $f_3+f_1$, $f_3+f_2$, and $f_3+g$, where $g$ is the multilinearization of $f_1f_2$. We also used the fact that $q,(1-2q)\geq 0$.
    This concludes the proof.
\end{proof}

In the proof of \cref{lem:close-to-uniform}, we will also use the following upper bound on the number of graphs that has adjacency matrix of a fixed rank over $\F_2$. For a graph $F$, we use $\rank(F)$ to denote the rank of its adjacency matrix over $\F_2$.

\begin{theorem}\cite[Theorem 3]{Mac69}\label{lemma:rank_bound} For a positive integer $s$ and an integer $a\ge 0$, the number of graphs $F$ on the vertex set $[s]$ such that $\rank(F) = a$ is at most $2^{sa-2}$. 
\end{theorem}

\begin{remark}
In \cite[Theorem 3]{Mac69}, it was shown that if $a$ is odd, then the number of graphs $F$ with $\rank(F) = a$ is $0$. If $a=2b$ is even, then the number of graphs $F$ with $\rank(F) = a$ is exactly
        \[\prod_{i=1}^{b}\frac{2^{2i-2}}{2^{2i}-1}
\prod_{i=0}^{2b-1}(2^{s-i}-1).\]
Note that we have $\prod_{i=1}^{b}2^{2i-2}=2^{b^2-b}$, $\prod_{i=1}^{b}(2^{2i}-1)\geq \prod_{i=1}^{b}2^{2i-1}=2^{b^2}$, and 
\[\prod_{i=0}^{2b-1}(2^{s-i}-1)\leq \prod_{i=0}^{2b-1}2^{s-i}=2^{2sb-2b^2+b}.\] 
Thus, we have
        \[\prod_{i=1}^{b}\frac{2^{2i-2}}{2^{2i}-1}
\prod_{i=0}^{2b-1}(2^{s-i}-1)\leq 2^{2sb-2b^2}\leq 2^{sa-2},\]
which is the upper bound stated in \cref{lemma:rank_bound}.
\end{remark}

\section{Proof of \cref{thm:main}}\label{sec:proof}

Recall that for two distributions $\mu_1, \mu_2$, we use $\dtv(\mu_1,\mu_2)$ to denote the total variation distance between $\mu_1,\mu_2$ defined as $\dtv(\mu_1,\mu_2) = \frac{1}{2}||\mu_1 - \mu_2||_1$. We use $G(U,p)$ to denote the Erd\H{o}s-Reny\'i random graph on the vertex set $U$ where each edge appears independently with probability $p$. We begin with the following key lemma that shows roughly the following: for $G\sim G(n,p)$ with vertex partition $V(G) = U\sqcup W$ where $W = \{w_1,\dots, w_r\}$. Let $G' = G * w_1 *\dots * w_r$, then the distribution on $G'[U]$ and $G(U,1/2)$ are close in total variation distance. We remark that the statement holds for any fixed subset $W$ without any additional randomness and every ordering of $W$.

\begin{lemma}\label{lem:close-to-uniform}
    For $p\in [0,1]$, $q\eqdef\min(p,1-p)$, and $G\sim G(n,p)$. For any vertex partition $V(G)=U\sqcup W$ with $\abs{U}=s$ and $\abs{W}=r$ satisfying $s\leq q^2r/6$. For any ordering on the vertices of $W=\{w_1,\dots,w_r\}$, let $G'=G* w_1*\dots* w_r$. Then the total variation distance between $G'[U]$ and $G(U,1/2)$ is at most $2^{-q^2r/6}$.
\end{lemma}

\begin{proof}
    Let $G_{\Delta}\eqdef G'[U]\triangle G[U]$. 
    Note that since $G_{\Delta}$ only depends on $G[U,W]$ and $G[W]$, it is independent from $G[U]$. Moreover, note that the symmetric difference between $G[U]\sim G(U,p)$ and an independent copy of $G(U,1/2)$ has the same distribution as $G(U,1/2)$. Indeed, this is because the indicator of each edge $e$ in the symmetric difference is of the form $X_e+Y_e$ where $X_e\sim \Ber(p)$ and $Y_e\sim \Ber(1/2)$ are independent, so $X_e+Y_e\sim \Ber(1/2)$. Thus, it suffices to show that the total variation distance between $G_{\Delta}$ and $G(U,1/2)$ is at most $2^{-q^2r/6}$. 
    
    Label the vertices in $U$ as $\{u_1,\dots,u_s\}=U$ and let $X \in\FF_2^{s\times r}$ be the $s\times r$ matrix with entries $X_{ij}$ being the indicator variable of $u_iw_j\in E(G)$ for all $i\in [s]$ and $j\in [r]$. Note that by definition, $X$ has independent $\Ber(p)$ entries corresponding to the biadjacency matrix of $G[U,W]$. For a square matrix $M$, we use $\OffDiag(M)\eqdef M-\textup{Diag}(M)$ to denote the matrix obtained by setting all the diagonal entries of $M$ to zero.
    \begin{claim}\label{claim:M}
        There is an invertible symmetric $r\times r$ matrix $M\in\FF_2^{r\times r}$ that only depends on $G[W]$ such that $\OffDiag(XMX^T)$ is the adjacency matrix of $G_{\Delta}$.
    \end{claim}
    \begin{proof}
        For each $i \in [r]$, let $G_i\eqdef G*w_1*\dots*w_i$. Let $z_i\in\FF_2^s$ be the indicator vector of the edges in $G_i[U,w_i]$ and let $X_i$ be the $i$-th column vector of $X$. Note that $z_i=X_i+\sum_j z_j$, where the summation is over all $j<i$ such that $w_iw_j$ is an edge in $G_{j-1}$. This is because for each $j<i$, local complementation at $w_j$ on $G_{j-1}$ flips the edge between $w_iu$ for $u\in U$ if and only if $w_iw_j$ and $w_ju$ are both edges in $G_{j-1}$. Consider the matrix $L\in\FF_2^{r\times r}$ defined as 
        \[L_{ji}=\begin{cases}
            1&\text{ if }j=i,\\
            1&\text{ if }j<i\text{ and }w_iw_j\text{ is an edge in }G_{j-1},\\
            0&\text{ otherwise}.
        \end{cases}\]
        Furthermore, note that $L$ is upper triangular with all diagonal entries being $1$, so $L$ is invertible. 
        Then $XL^{-1}$ is the matrix with columns exactly $(z_i)_{i=1}^r$. Now, we claim that $M=L^{-1}(L^{-1})^T$ satisfies the conditions in the claim. First, it is clear that $L$ only depends on $G_{j-1}[W]$, which only depends on $G[W]$. Thus, $M$ only depends on $G[W]$.

        Note that from the definition of local complementation, the edge $u_ju_k$ is flipped during local complementation at $w_i$ on $G_{i-1}$ if and only if both $u_jw_i$ and $u_kw_i$ are edges in $G_{i-1}$. Thus, the adjacency matrices of $G_{i-1}[U]$ and $G_i[U]$ differ by exactly $\OffDiag(z_iz_i^T)$. It follows directly that the adjacency matrix of $G_{\Delta}$ is exactly $\OffDiag(\sum_{i=1}^rz_iz_i^T)=\OffDiag(XMX^T)$. This completes the proof.
    \end{proof}
    
    Next, we analyze the distribution of $G_{\Delta}$ using discrete Fourier transformation. To simplify the analysis, we will always condition on $G[W]=G^*$ from now on. We use $\mu:\FF_2^{\binom{U}{2}}\rightarrow [0,1]$ to denote the probability distribution of $G_{\Delta}$ conditioning on $G[W]=G^*$. Namely, for each $H\subseteq \binom{U}{2}$, we define
    \[\mu(H)\eqdef\PP(G_{\Delta}=H\mid G[W]=G^*).\]
    In addition, let $\mu_{\textup{unif}}$ be the uniform distribution over $\binom{U}{2}$, i.e. $\mu_{\textup{unif}}=2^{-\binom{s}{2}}$ where $s = |U|$ for every $H\subseteq \binom{U}{2}$.
    
    For a function $f:\FF_2^{\binom{U}{2}}\rightarrow \RR$, we define its Fourier transformation $\widehat{f}:\FF_2^{\binom{U}{2}}\rightarrow \RR$ by
    \[\widehat{f}(F)\eqdef\sum_{H\subseteq \binom{U}{2}}f(H)(-1)^{\abs{H\cap F}}.\] Here we abuse notation and use $H$ and $F$ to denote the graphs on the vertex set $U$ and their corresponding edge indicator vector in $\F_2^{\binom{U}{2}}$. We first recall the following standard fact of bounding the total variation distance between $\mu$ and $\mu_{\textup{unif}}$ via the Fourier coefficients of $\mu$.
    
    \begin{claim}\label{claim:Fourier}
        Let $\mu$ be a probability distribution over $\binom{U}{2}$ and $\mu_{\textup{unif}}$ be the uniform distribution over $\binom{U}{2}$. Then we have
        \[d_{\textup{TV}}(\mu,\mu_{\textup{unif}})\le \frac{1}{2}\sum_{F\ne\varnothing\subseteq \binom{U}{2}}|\widehat{\mu}(F)|.\]
    \end{claim}
    \begin{proof}
        Note that we have 
        \[\widehat{\mu_{\textup{unif}}}(F)=\begin{cases}
            1&\text{ if }F=\varnothing,\\
            0&\text{ otherwise},
        \end{cases}\]
        and $\widehat{\mu}(\varnothing)=1$. Thus, by linearity of the Fourier transform, we have
        \[\widehat{(\mu-\mu_{\textup{unif}})}(F) = \widehat{\mu}(F) - \widehat{\mu_{\textup{unif}}}(F)=\begin{cases}
            0&\text{ if }F=\varnothing,\\
            \widehat{\mu}(F)&\text{ otherwise}.
        \end{cases}\]
        Thus, recall that $d_{\textup{TV}}(\mu,\mu_{\textup{unif}})\eqdef\frac{1}{2}||\mu-\mu_{\textup{unif}}||_1$, the claim follows from the fact that $||f||_1\leq \sum_{F\subseteq\binom{U}{2}}\abs{\widehat{f}(F)}$. For completeness, we include a proof of this standard fact here. By the Fourier inversion formula, we have $f(H)=2^{-\binom{s}{2}}\sum_{F\subseteq\binom{U}{2}}\widehat{f}(F)(-1)^{\abs{H\cap F}}$, and hence
        \[\abs{f(H)}\leq 2^{-\binom{s}{2}}\sum_{F\subseteq\binom{U}{2}}\abs{\widehat{f}(F)}.\]
        Therefore, 
        \[||f||_1=\sum_{H\subseteq\binom{U}{2}}\abs{f(H)}\leq 2^{-\binom{s}{2}}\sum_{H\subseteq\binom{U}{2}}\sum_{F\subseteq\binom{U}{2}}\abs{\widehat{f}(F)}=\sum_{F\subseteq\binom{U}{2}}\abs{\widehat{f}(F)}.\qedhere\]
    \end{proof}
    
    Now we upper bound $\abs{\widehat{\mu}(F)}$ for each $F\subseteq \binom{U}{2}$ by the rank of its adjacency matrix. From now on, we abuse notation and use $\rank(F)$ to denote the rank of the adjacency matrix of $F$ over $\F_2$.
    
    \begin{claim}\label{claim:Fourier_bound}
        For any $F\subseteq \binom{U}{2}$, we have 
        \[\abs{\widehat{\mu}(F)}\leq (1-2q^2)^{r\cdot\rank(F)/6-1}.\]
    \end{claim}
    
    \begin{proof}
        Recall that by definition, we have
        \[\widehat{\mu}(F)=\sum_{H\subseteq\binom{U}{2}}\PP(G_{\Delta}=H\mid G[W]=G^*)(-1)^{\abs{H\cap F}}=\EE_{H\sim \mu}[(-1)^{\abs{H\cap F}}],\] 
        where we recall that $\mu$ is the distribution of $(G_{\Delta}\mid G[W]=G^*)$. By $\cref{claim:M}$, we know that there exists an invertible symmetric matrix $M$ satisfying the following: Let $X\in \F_2^{s\times r}$ be the biadjacency matrix of $G[U,W]$ over $\F_2$ with independent $\Ber(p)$ entries. Let $v_i\eqdef(X_{i1},\dots,X_{ir})$ be the $i$-row of $X$. Then for any $\{i,j\}\in \binom{[s]}{2}$, we have $u_iu_j\in H$ if and only if $v_iMv_j^T=1$.
        
        Define for convenience $F_0\eqdef\{\{i,j\}\in\binom{[s]}{2}\mid u_iu_j\in F\}$ so that $F_0$ is the graph $F$ relabeled on the vertex set $[s]$ via the ordering $u_i\mapsto i$ on $U$. It follows that $\abs{H\cap F}\equiv \sum_{\{i,j\}\in F_0}v_iMv_j^T\pmod{2}$. Therefore, we can rewrite
        \[\widehat{\mu}(F)=\EE_X\left[(-1)^{\sum_{ij\in F_0}v_iMv_j^T}\right].\]

        Let $f\in \FF_2[X_{11},\dots,X_{sr}]$ be the degree-$2$ polynomial given by $f(X)=\sum_{\{i,j\}\in F_0}v_iMv_j^T$. Note that $f$ is multilinear since there are no self-loops in $F_0$. Furthermore, for any $i,k\in [s]$ and $j,\ell\in [r]$, the term $X_{ij}X_{k\ell}$ appears in $f$ if and only if $\{i,k\}\in F_0$ and $M_{j\ell} = 1$. Thus, we know that $\Coef_2(f)=A_{F_0}\otimes M$, where $A_{F_0}$ is the adjacency matrix of $F_0$. Namely, the coefficient of the term $X_{ij}X_{k\ell}$ in $f$ is $(A_{F_0})_{ik}M_{j\ell}$. Therefore, we have $\rank(\Coef_2(f))=\rank(M)\cdot \rank(F_0) = r\cdot\rank(F)$ since $F_0$ is just $F$ on the vertex set $[s]$. By \cref{lem:1}, we can conclude that
        \[\abs{\widehat{\mu}(F)}=\abs{\EE_X\left[(-1)^{f(X)}\right]}\leq (1-2q^2)^{\lfloor r\cdot\rank(F)/6\rfloor}\leq (1-2q^2)^{r\cdot\rank(F)/6-1}.\qedhere\]
    \end{proof}

    Now for any $a\ge 0$, we have by \cref{thm:main} the number of graphs $F\subseteq \binom{U}{2}$ with $\rank(F) = a$ is at most $2^{sa-2}$. Combined with \cref{claim:Fourier,claim:Fourier_bound}, we obtain
    \begin{align*}
        d_{\textup{TV}}(\mu,\mu_{\textup{unif}})\leq &\frac{1}{2}\sum_{F\ne\varnothing\subseteq \binom{U}{2}}|\widehat{\mu}(F)|
        \leq \frac{1}{2}\sum_{F\ne\varnothing\subseteq \binom{U}{2}}(1-2q^2)^{r\cdot\rank(F)/6-1}
        \leq \frac{1}{2}\sum_{a=1}^\infty 2^{sa-2}(1-2q^2)^{ra/6-1}\\
        \leq &\frac{1}{2}\sum_{a=1}^\infty 2^{sa}2^{-2q^2ra/6},
    \end{align*}
    where the last inequality follows the fact that $1-2q^2\geq 1/4$ and for all $x\geq 0$ we have $1-x\leq 2^{-x}$. Note that the common ratio in the geometric series is $2^{s-2q^2r/6}\le 2^{-s}\le 1/2$ since $s-2q^2r/6\leq -s\leq -1$. Therefore, we have
    \[\frac{1}{2}\sum_{a=1}^\infty 2^{sa-2q^2ra/6}\leq 2^{s-2q^2r/6}\leq 2^{-q^2r/6}.\]
    By an averaging argument over all possible assignments $G^*$ to $G[W]$, we can conclude that
    \[d_{\textup{TV}}(G_{\Delta},G(U,1/2))\leq 2^{-q^2r/6}\]
    as desired.
\end{proof}

Now, we can prove \cref{thm:main}.
\begin{proof}[Proof of \cref{thm:main}]
    It suffices to prove the statement for $k=\lfloor q\sqrt{n}/100\rfloor$ since $k$-vertex-minor universality is monotone decreasing in $k$. Let $G\sim G(n,p)$. For any fixed $S\subseteq [n]$ with size $k$. We partition the vertices $V(G)=U\sqcup W$ with $S\subseteq U$, $\abs{U}=\lfloor q^2n/12\rfloor$, and $W=\{w_1,\dots,w_r\}$ with $r=n-\lfloor q^2n/12\rfloor$. When $n$ is sufficiently large, we have $r\geq n/2$. Therefore, we have $q^2r/6\geq q^2n/12\geq \lfloor q^2n/12\rfloor=s$. Let $G'=G* w_1*\dots* w_r$. By \cref{lem:close-to-uniform}, we have the total variation distance between $G'[U]$ and $G(U,1/2)$ is at most $2^{-q^2n/12}$. Therefore, there is a coupling such that 
    \[\PP(G'[U]\neq G(U,1/2))\leq 2^{-q^2n/12}.\]
    Applying \cref{thm:1/2} to $U$ with $\abs{U}=\lfloor q^2n/12\rfloor$ and pick $C=2\log_2(4/3)\lfloor q^2n/12\rfloor/k^2-1$. When $n$ is sufficiently large, we know that $C> q^2n/15k^2$, and hence we have
    \[\PP(G(U,1/2)\text{ is not vertex-minor universal})\le 2^{-(1+o(1))Ck^2/2}\leq 2^{-q^2n/40}.\]
    Since we have
    \[\PP(G'[U]\text{ is not vertex-minor universal})\le \PP(G'[U]\neq G(U,1/2)) + \PP(G(U,1/2)\text{ is not vertex-minor universal}),\]
    we can conclude that for $n$ sufficiently large, we have
    \[\PP(G'[U]\text{ is not vertex-minor universal})\leq 2^{-q^2n/12}+2^{-q^2n/40}\le 2^{1-q^2n/40}.\]
    
    For each $S\subseteq [n]$ with size $k$, we fix a partition $V(G)=U\sqcup W$ as above and we have the probability that there is a graph on vertex set $S$ that is not a vertex-minor of $G$ is at most $2^{1-q^2n/40}$. Now, we take a union bound over all possible $S$. Since there are $\binom{n}{k}\leq n^k$ many subsets $S$ and by assumption we have $k=\lfloor q\sqrt{n}/100\rfloor$, the number of subsets $S$ is at most $n^{\lfloor q\sqrt{n}/100\rfloor}\leq 2^{q\sqrt{n}\log n}\leq 2^{q^2n/100}$. Therefore, by the union bound, $G$ is not $k$-vertex-minor universal with probability at most $2^{q^2n/100}\cdot 2^{1-q^2n/40}\leq 2^{-q^2n/100}$ for $n$ sufficiently large. This completes the proof.
\end{proof}

\section{Concluding remarks}\label{sec:conclude}

In \cite{AFFMP26}, the authors also considered the universality question for pivot-minors and matroid-minors and proved analogous results to \cref{thm:1/2} for $G(n,1/2)$ in these settings. We believe our methods should be able to generalize those results to the setting of $G(n,p)$. We briefly outline how to adapt our strategy to those settings and discuss the extra obstruction in this section.

Since minors in binary matroids can be represented by pivot-minors of their fundamental graphs in the ordered bipartite setting, let us follow the setup in \cite[Section 4]{AFFMP26} and work with ordered bipartite random graph from now on. See \cite{AFFMP26} for more details on the reduction from binary matroids to pivot minors. For a simple graph $G$ and an edge $uv\in E(G)$, a pivot at $uv$, denoted by $G\times uv$, is defined to be the operation $G*u*v*u$. One can check that $G *u*v*u = G*v*u*v$, so the pivot operation does not depend on the ordering of $uv$ and thus is well-defined. Equivalently, the graph $G\times uv$ is obtained by taking the symmetric difference between $G$ with the complete tripartite graph with parts $N(v)\cap N(u)$, $N(v)\setminus (N(u)\setminus \{u\})$, and $N(u)\setminus (N(v)\setminus \{v\})$, and then swapping the labels of $u$ and $v$. An ordered bipartite graph $G = (L, R, E)$ has bipartition $V(G) = L\sqcup R$ and edges contained in $L\times R$. Note that for $uv\in E(G)$, we have 
\[G\times uv \eqdef (L\triangle \{u,v\}, R \triangle \{u,v\}, E') \]
where $E'$ is the edges of the unordered bipartite graph $(L,R,E)\times uv$. For any ordered bipartite graph $H = (V_L,V_R,E(H))$ with $V_L\cup V_R\subseteq V(G)$, we say that $H$ is a (labeled) pivot-minor of $G$ if it can be obtained on $V_L\cup V_R$ via a sequence of vertex deletions and pivots. We say that ordered bipartite graph $G$ is \emph{$k$-pivot-minor universal} if any ordered bipartite graph $H$ on any $k$ vertices of $V(G)$ is a pivot minor of $G$.

Let $G(a, b,p)$ denote the random ordered bipartite graph on vertex bipartition $L\sqcup R$ with $|L| = a$, $|R| = b$, and edges between $L\times R$ present independently with probability $p$. We believe \cref{lem:close-to-uniform} can be adapted in this setting with the following strategy. 

We fix disjoint subsets $U_L\sqcup W_L\subseteq L$ and $U_R\sqcup W_R\subseteq R$ such that $|U_L| = |U_R|=s$ and $|W_L| =|W_R| = r$. 
Suppose that we can find disjoint pairs $w_1w'_1,\dots,w_{r'}w'_{r'}\in W_L\times W_R$ such that $w_{i+1}w'_{i+1}$ is an edge in $G_{i-1}=G\times w_1w'_1\times\dots\times w_{i}w'_{i}$ for all $i=0,1,\dots,r'-1$. Then we set $G'=G_{r'}$ to be the resulting graph, and let $W'_L=\{w_1,\dots,w_{r'}\}, W'_R=\{w'_1,\dots,w'_{r'}\}$.

Define $G_\Delta = G'[U_L,U_R]\triangle G[U_L,U_R]$. Let $G_i = G\times w_1w'_1\times \dots \times w_iw'_i$ and $U_L = \{u_1,\dots, u_{s_1}\}$, $U_R = \{u'_1,\dots, u'_{s_2}\}$. Then we can define $X_L\in \F_2^{|U_L|\times |W'_R|}$ with entries $(X_L)_{ij}$ being the indicator variable of the edge $u_iw'_j$. Similarly, define $X_R\in \F_2^{|U_R|\times |W'_L|}$ with entries $(X_R)_{ij}$ being the indicator variable of the edge $u'_iw_j$. Then similar to \cref{claim:M}, we claim that there is an invertible (not necessarily symmetric) matrix $M\in \F_2^{r'\times r'}$ only depending on $G[W_L,W_R]$ such that $\OffDiag(X_L M X_R^T)$ is the adjacency matrix of $G_\Delta$.

Indeed, we can define $z_i^L\in \F_2^{U_L}$ to be the indicator vector of the edges in $G_i[U_L,w'_i]$ and define $z_i^R\in \F_2^{U_R}$ to be the indicator vector of the edges in $G_i[U_R,w_i]$. The $i$-th pivot changes adjacency matrix of $G_{i-1}[U_L,U_R]$ by exactly $z_i^L(z_i^R)^T$. We can find upper triangular matrices $Q_L$ and $Q_R$ such that $X_LQ_L$ is the matrix with columns $(z_i^L)_{i = 1}^{r'}$ and $X_RQ_R$ is the matrix with columns $(z_i^R)_{i = 1}^{r'}$. Taking $M = Q_LQ_R^T$ satisfies the desired conditions. The rest of the argument should follow through, and we can upper bound the total variation distance between $G'[U_L,U_R]$ and $G'(s,s,1/2)$ by $2^{-\Omega(q^2r')}$ where $q=\min(p,1-p)$.

Thus, in order to prove an analogous result for pivot-minor universality, it is sufficient to answer the following question.
\begin{question}
    Let $G\sim G(r,r,p)$ and $\min(p,1-p)=\Omega(\log n/\sqrt{n})$. Can we find disjoint pairs of vertices $w_1w'_1,\dots,w_{r'}w'_{r'}$ such that $w_{i+1}w'_{i+1}$ is an edge in $G\times w_1w'_1\times\dots\times w_{i}w'_{i}$ for all $i=0,1,\dots,r'-1$ and $r'=\Theta(r)$ with high probability?
\end{question}
We believe that the answer to this question is affirmative since $p$ is much larger than the threshold of having a perfect matching.

After we put the first version of this note on arXiv, the following solution to the question is pointed out to us by Mathieu Rundstr\"om. The answer is to this question is affirmative, and it is true that $r'\geq r/2$ with probability at least $1-2^{-\Omega(qr)}$, where $q=\min(p,1-p)$. This can be proved with the following observation. Let $A$ be the biadjacency matrix of $G$, then the maximum possible $r'$ is at least $\gamma\eqdef\rank(A)$. Here, we recall that $\rank(A)$ denotes the rank over $\FF_2$. To see this observation, we can restrict $A$ to an invertible $\gamma\times\gamma$ submatrix $B$. Let $V_B$ be the vertex set corresponding to the rows and columns of $B$. It is sufficient to pick $w_1w'_1,\dots,w_{\gamma}w'_{\gamma}$ from $V_B$.

We can show this by inducting on $\gamma$.
We first pick an arbitrary entry with value one, say $B_{11}=1$, and set $w_1w'_1$ to be the edge correspond to this entry. Suppose $B=\begin{pmatrix}
    1&\alpha\\
    \beta &B'
\end{pmatrix}$, where $\alpha\in \FF_2^{1\times (\gamma-1)}$, $\beta\in \FF_2^{(\gamma-1)\times 1}$, and $B'\in \FF_2^{(\gamma-1)\times (\gamma-1)}$. It follows that the biadjacency matrix of $(G\times w_1w'_1)[V_B]$ equals
$\begin{pmatrix}
    1&0\\
    0 &B'+\alpha\beta
\end{pmatrix}$. Note that $B'+\alpha\beta$ is invertible since 
\[\begin{pmatrix}
    1&0\\
    0 &B'+\alpha\beta
\end{pmatrix}=\begin{pmatrix}
    1&0\\
    \beta &I_{\gamma-1}
\end{pmatrix}\begin{pmatrix}
    1&\alpha\\
    \beta &B'
\end{pmatrix}\begin{pmatrix}
    1&\alpha\\
    0 &I_{\gamma-1}
\end{pmatrix}\]
is invertible. Thus, we can find $w_2w'_2,\dots,w_\gamma w'_\gamma$ by applying the inductive hypothesis to $B'+\alpha\beta$.

Now, it is sufficient to show that with high probability $\rank(A)\geq r/2$. Note that $A\in\FF_2^{r\times r}$ is a random binary matrix with i.i.d. $\Ber(p)$ entries. Let $A_i$ be the $i$-th column vector of $A$. We first bound the probability $\PP[A_{i+1}\in\Span(A_1,\dots,A_{i})\mid A_1,\dots,A_{i}]$. Let $\cV_{i}=\Span(A_1,\dots,A_{i})$ and suppose the dimension of $\cV_{i}$ is $d_{i}$. Given $A_1,\dots,A_{i}$, we can find a set $I_i$ of $d_{i}$ entries such that the projection $\pi_{I_i}(\cV_{i})$ onto $I_i$ is surjective. Hence, for any $v\in \cV_{i}$, the projection $\pi_{[r]\setminus I_i}(v)$ is uniquely determined by $\pi_{I_i}(v)$. Indeed, we know that there is a linear map $T$, depending on $\cV_i$, such that $\pi_{[r]\setminus I_i}(v)=T(\pi_{I_i}(v))$. Thus, $A_{i+1}\in \cV_i$ only if $\pi_{[r]\setminus I_i}(A_{i+1})=T(\pi_{I_i}(A_{i+1}))$, which happens with probability at most $(1-q)^{r-d_i}$ since we can first expose $\pi_{I_i}(A_{i+1})$. Hence, we can conclude that
\[\PP[A_{i+1}\in \cV_i\mid A_1,\dots,A_{i}]\leq (1-q)^{r-d_i}.\]
Fix $0<\gamma_0<r$. Let $\mathbbm{1}_{i}$ be the indicator function of the event that $A_{i}\in\cV_{i-1}$ and $\dim\cV_{i-1}\leq \gamma_0$ for $i=1,\dots,r$, assuming that $\cV_0=\{0\}$. Set $Z=\sum_{i=1}^r \mathbbm{1}_{i}$ to be the number of indices $i$ such that $A_{i}\in\cV_{i-1}$ happens before the rank exceeds $\gamma_0$. We know that $\EE[\mathbbm{1}_{i}]\leq (1-q)^{r-d_i}\leq (1-q)^{r-\gamma_0}$ holds for all $i=1,\dots,r$. Thus, by linearity of expectation we have $\EE [Z]\leq r(1-q)^{r-\gamma_0}$. On the other hand, we know that $\rank(A)\leq \gamma_0$ only if $Z\geq r-\gamma_0$. Thus,
\[\PP[\rank(A)\leq \gamma_0]\leq\PP[Z\geq r-\gamma_0]\leq \frac{\EE[Z]}{r-\gamma_0}\leq \frac{r(1-q)^{r-\gamma_0}}{r-\gamma_0}.\]
If we set $\gamma_0=r/2$, then the probability is at most $2(1-q)^{r/2}=2^{-\Omega(qr)}$. This proves the claim that $r'\geq r/2$ with probability at least $1-2^{-\Omega(qr)}$.

From all the discussion above, we can conclude in a way similar to \cref{sec:proof} that the bipartite graph $G$ is pivot-minor universal with high probability when $\abs{L},\abs{R}$ are not too unbalanced. To be more explicit, suppose $\abs{L}=a, \abs{R}=b$ with $a\geq b$, and $G\sim G(L,R,p)$. We can show that $G$ is $k$-pivot-minor universal with high probability if $k=O(\min(q\sqrt{b},q^2b/\log a))$. This can be done by picking $s=\Theta(q^2b)$ and $r=\abs{R}-s=\Theta(b)$ in the discussion above. We know that for any fixed ordered bipartite graph $H = (V_L,V_R,E(H))$ with $V_L\cup V_R\subseteq V(G)$ and $\abs{V_L\cup V_R}=k$, the probability that $H$ is not a pivot minor of $G$ is $O(2^{-q^2b})$ since $k=O(q\sqrt{b})$. Because there are at most $2^{O(k\log a)}$ many choices for $(V_L,V_R)$, we can conclude by the union bound that the probability that $G$ is not $k$-pivot-minor universal is at most $2^{O(k\log a)-\Omega(q^2b)}=o(1)$.

\printbibliography
\appendix
\section{Proof of \cref{thm:openai}}\label{appendix:openai}
In this appendix, we prove \cref{thm:openai}. Let $p\in (0,1/2]$, we will first prove the case where $G\sim G(n,p)$, and the case $G\sim G(n,1-p)$ will follow similarly.

The key step is to find the structure defined as follows.

\begin{definition}
Let $G$ be a graph and let $S \subseteq V(G)$.  A \emph{clean subdivision $R$ of $S$}
is a subset of vertices in $V(G)\setminus S$ that contains distinct vertices $a_{uv}, b_{uv}\notin S$ for every $\{u,v\} \in \binom{S}{2}$ such that the following holds. For each $\{u,v\} \in \binom{S}{2}$, the vertices $u,a_{uv},b_{uv},v$ form a path of length three, and these are the only edges in $G[S\cup R]\setminus G[S]$.
\end{definition}

It is not hard to show that any graph on $S$ is a vertex minor of $G$ if $S$ admits a clean subdivision.
\begin{lemma}
\label{lem:toggling}
Let $G$ be a graph, $S\subseteq V(G)$, and suppose that $G$ contains a clean subdivision $R$ of $S$. Then for every graph $H$ with $V(H)=S$, we know that $H$ is a vertex minor of $G$.
\end{lemma}
\begin{proof}
We may first remove all vertices beside $R\cup S$ from $V(G)$.
For each $\{u,v\} \in \binom{S}{2}$, a local
complementation first at $a_{uv}$ and then at $b_{uv}$ flips the edge $\{u,v\}$. Moreover, this operation only change edges between $u,v,a_{uv},b_{uv}$. Thus, we may apply this operation in any order to pairs $\{u,v\} \in G[S]\triangle H[S]$. After all the operations are done, we remove all vertices in $R$, and the remaining graph is exactly $H$.
\end{proof}

Thus, it is sufficient to show that with high probability that every $S\subseteq V(G)$ of size $k$ admits a clean subdivision. To do so, we will first show that, with high probability, one can greedily pick $a_{uv},b_{uv}$ for each $\{u,v\}\in \binom{S}{2}$.

We say that a graph $G$ is \emph{good} if the following holds. For every $\{u,v\}\subseteq S \subseteq F \subseteq [n]$ with $\abs{S}=k$ and $|F| \le k^2$, there always exist two vertices $a_{uv},b_{uv}\notin F$ such that $u,a_{uv},b_{uv},v$ form a path of length three, and these are the only edges between $\{a_{uv},b_{uv}\}$ and $F$.

\begin{lemma}
\label{lem:sparse-reservoir}
Assume that $n$ is sufficiently large, $1/\sqrt{n}\leq p\leq n^{-1/3}/2$, and $G\sim G(n,p)$.
Then $G$ is good with probability at least $1-e^{-pn/200}$.
\end{lemma}
\begin{proof}
    We will union bound over all possible triples $(\{u,v\},S,F)$ where $\{u,v\}\subseteq S \subseteq F \subseteq [n]$, $\abs{S}=k$, and $|F| \le k^2$. We first fix such a triple $(\{u,v\},S,F)$ and let $r=\abs{F}$ for convenience.
    Let $A_u$ be the set of vertices $x\in V(G)\setminus F$ such that $u$ is the only neighbor of $x$ in $F$, and we define $A_v$ similarly.
    For each $x\in V(G)\setminus F$, we know that $x\in A_u$ with probability $\theta=p(1-p)^{r-1}$. Therefore, $\abs{A_u}\sim\textup{Bin}(n-r,\theta)$. We know that the expectation of $\abs{A_u}$ is $\mu\eqdef(n-r)p(1-p)^{r-1}$. Since $r\leq k^2$, we know that $(1-p)^{r-1}\geq 1-k^2p\geq 1-p^3n\geq 1/2$ by Bernoulli's inequality and our assumption on $p$ and $k$. We also have $n-r\geq n-k^2\geq n/2$. Thus, we get $\mu\geq pn/4$. By the multiplicative form of Chernoff's inequality, we have
\[
\PP\bigl(|A_u| < \mu/2\bigr) \le e^{-\mu/8}
\quad\text{and}\quad
\PP\bigl(|A_v| < \mu/2\bigr) \le e^{-\mu/8}.
\]

Conditioning on $A_u,A_v$, the probability that there is no edge between $A_u,A_v$ is $(1-p)^{\abs{A_u}\cdot\abs{A_v}}$. Note that this is at least $(1-p)^{\mu^2/4}\leq e^{-p\mu^2/4}$ if both $\abs{A_u}$ and $\abs{A_v}$ are at least $\mu/2$. Therefore, the probability that the triple $(\{u,v\},S,F)$ fails to find the vertices $a_{uv},b_{uv}$ is at most 
\[2e^{-\mu/8}+e^{-p\mu^2/4}\leq 2e^{-pn/32}+e^{-p^3n^2/64}\leq e^{-pn/100}\]
when $n$ is sufficiently large.
Note that there are at most 
\[
\binom{n}{k} \binom{k}{2}\sum_{r=k}^{k^2} \binom{n}{r}
 \le k^4n^{k+k^2}\leq n^{2k^2}
\]
such triples. Therefore, the probability that $G$ is good is at least
\[1-n^{2k^2} e^{-pn/100}\geq 1-e^{-pn/200}\]
when $n$ is sufficiently large. Here, we used the fact that $2k^2\log n\leq 2p^2n\log n\leq pn/200$.
\end{proof}

Now, it remains to show that we can inductively find a clean subdivision for every $S$ if $G$ is good.

\begin{lemma}\label{lem:good}
    If $G$ is good, then for any $S\subseteq V(G)$ of size $k$, there exists a clean subdivision of $S$ in $G$.
\end{lemma}
\begin{proof}
    For any $S\subseteq V(G)$ with $\abs{S}=k$, we build its clean subdivision inductively. We fix an arbitrary order $\{u_1,v_1\},\dots,\{u_{\binom{k}{2}},v_{\binom{k}{2}}\}$ of $\binom{S}{2}$. In the $i$-th round, suppose that we already found $a_{u_jv_j},b_{u_jv_j}$ for each $j<i$. We may use the definition of $G$ being good with $u_i,v_i$ in place of $u,v$ and $F_i=\{u_j,v_j,a_{u_jv_j},b_{u_jv_j}\mid j<i\}$ in place of $F$. Note that $\abs{F_i}\leq k+2\binom{k}{2}=k^2$. Thus, we can conclude that there exist $a_{u_iv_i},b_{u_iv_i}$ such that $u_i,a_{u_iv_i},b_{u_iv_i},v_i$ form a path of length three, and these are the only edges between $\{a_{u_iv_i},b_{u_iv_i}\}$ and $F_i$. After running all $\binom{k}{2}$ rounds of this process, we can conclude that $S$ admits a clean subdivision.
\end{proof}

Combining all the lemmas above, we get \cref{thm:openai} \ref{thm:openai_part_a} immediately.
\begin{proof}[Proof of \cref{thm:openai} \ref{thm:openai_part_a}]
    From \cref{lem:sparse-reservoir}, we know that $G$ is good with probability at least $1-e^{-pn/200}$. From \cref{lem:good,lem:toggling}, we know that a good graph $G$ is vertex minor universal. This completes the proof.
\end{proof}

It remains to prove \cref{thm:openai} \ref{thm:openai_part_b} where $G\sim G(n,1-p)$. For any set $S$, by applying \cref{lem:sparse-reservoir} to the complement $\overline{G}\sim G(n,p)$, we can find a clean subdivision $R$ of $S$ in $\overline{G}$ with high probability. If there exists a common neighbor $z_F$ of all the vertices in $F=R\cup S$, then we can apply a local complementation at $z_F$, so that $R$ is a clean subdivision of $S$ in $G* z_F$.
\begin{lemma}\label{lem:common_neighbor}
Assume that $n$ is sufficiently large, $1/\sqrt{n}\leq p\leq n^{-1/3}/2$, and $G\sim G(n,1-p)$.
Then the following holds with probability at least $1-e^{-n/8}$. For any set $F\subseteq V(G)$ with $\abs{F}=k^2$, there exists a common neighbor $z_F$ of all the vertices in $F$.
\end{lemma}
\begin{proof}
    We union bound over all possible $F$. For any fixed $F$ and any $x\in V(G)\setminus F$, the probability that $x$ is a common neighbor of all the vertices in $F$ is $(1-p)^{k^2}\geq 1-k^2p\geq 1-p^3n\geq 1/2$. Therefore, with probability at most $2^{-(n-k^2)}\leq e^{-n/4}$, there does not exist a common neighbor $z_F$ of all the vertices in $F$. Note that there are $\binom{n}{k^2}\leq n^{k^2}$ such $F$, so the probability that $z_F$ exists for all $F$ is at least
    \[1-n^{k^2}e^{-n/4}\geq 1-e^{-n/8}.\]
    Here, we used the fact that $k^2\log n\leq p^2n\log n\leq n/8$ when $n$ is large enough.
\end{proof}

Now, we are ready to prove \cref{thm:openai} \ref{thm:openai_part_b}.
\begin{proof}[Proof of \cref{thm:openai} \ref{thm:openai_part_b}]
From \cref{lem:common_neighbor,lem:sparse-reservoir}, the probability that both $G$ satisfying the conclusion in \cref{lem:common_neighbor} and $\overline{G}$ being good is at least $1-e^{-pn/200}-e^{-n/8}\geq 1-2e^{-pn/200}$. Indeed, these two conditions imply that $G$ is vertex minor universal. For any set $S\subseteq V(G)$ of size $k$, there exists a clean subdivision $R$ of $S$ in $\overline{G}$. By the conclusion of \cref{lem:common_neighbor}, there exists a common neighbor $z_F$ of all vertices in $F=R\cup S$. We know that $R$ is also a clean subdivision of $S$ in $G*z_F$. Thus, by \cref{lem:toggling}, we know that any graph on the vertex set $S$ is a vertex minor of $G*z_F$, and hence a vertex minor of $G$. This completes the proof.
\end{proof}

\end{document}